\documentclass[11pt]{article}

\usepackage{oldlfont}
\usepackage{amsfonts}
\usepackage{graphicx}
\usepackage{eepic}

\newcommand{\be}{\begin{equation}}
\newcommand{\ee}{\end{equation}}
\newcommand{\bea}{\begin{eqnarray}}
\newcommand{\eea}{\end{eqnarray}}
\newcommand{\beas}{\begin{eqnarray*}}
\newcommand{\eeas}{\end{eqnarray*}}
\newcommand{\ba}{\begin{array}}
\newcommand{\ea}{\end{array}}

\newcommand{\<}  {\langle}
\renewcommand{\>}{\rangle}
\newcommand{\field}[1]{\mathbb{#1}}
\newcommand{\supp}{{\mathrm supp}}

\newcommand{\grad}{\nabla}
\renewcommand{\div}{{\mathrm div}}
\newcommand{\curl}{{\mathrm curl}}

\newcommand{\G}{\Gamma}
\newcommand{\g}{\gamma}
\newcommand{\eps}{\varepsilon}

\newcommand{\bolda}{{\mathbf a}}

\newcommand{\bn}{{\mathbf n}}

\newcommand{\bu}{{\mathbf u}}
\newcommand{\bv}{{\mathbf v}}
\newcommand{\bw}{{\mathbf w}}
\newcommand{\bx}{{\mathbf x}}

\newcommand{\bX}{{\mathbf X}}
\newcommand{\bH}{{\mathbf H}}
\newcommand{\bL}{{\mathbf L}}
\newcommand{\bcurl}{{\mathbf curl}}
\newcommand{\bzero}{{\mathbf 0}}

\newcommand{\bnu}{\hbox{\mathversion{bold}$\nu$}}

\newcommand{\bphi}{\hbox{\mathversion{bold}$\phi$}}

\newcommand{\CE}{{\cal E}}

\newcommand{\CI}{{\cal I}}
\newcommand{\CL}{{\cal L}}

\newcommand{\CP}{{\cal P}}
\newcommand{\bCP}{\hbox{\mathversion{bold}$\cal P$}}

\newcommand{\CT}{{\cal T}}

\newcommand{\stack}[2]{\mathrel{\mathop{#2}\limits_{\scriptstyle #1}}}
\newcommand{\tH}{\tilde H}

\newcommand{\ttu}{\widetilde{\widetilde{u^\circ}}}
\newcommand{\ttv}{\widetilde{\widetilde{v^\circ\,}}}

\newtheorem{theorem}{Theorem} [section]
\newtheorem{lemma}{Lemma} [section]

\newtheorem{prop}{Proposition} [section]
\newtheorem{remark}{Remark} [section]
\newenvironment{proof}{\noindent\textbf{Proof.}\ }
              {\nopagebreak\hbox{ }\hfill$\Box$\bigskip}
\newcommand{\qed}{\nopagebreak\hbox{ }\hfill$\Box$\bigskip}

\title{A new $\bH(\div)$-conforming $p$-interpolation operator in two dimensions
       \thanks{Supported by EPSRC under grant no. EP/E058094/1.}}

\author{Alexei Bespalov
\thanks{Department of Mathematical Sciences, Brunel University,
        Uxbridge, West London UB8 3PH, UK.
        Email: {\tt albespalov@yahoo.com}}
        \and
        Norbert Heuer
\thanks{Facultad de Matem\'aticas, Pontificia Universidad Cat\'olica de Chile,
        Avenida Vicu\~na Mackenna 4860, Santiago, Chile.
        Email: {\tt nheuer@mat.puc.cl}}
        }
\begin{document}
\date{}
\maketitle

\begin{abstract}
In this paper we construct a new $\bH(\div)$-conforming projection-based
$p$-interpolation operator that assumes only $\bH^r(K) \cap \tilde\bH^{-1/2}(\div,K)$-regularity
($r > 0$) on the reference element (either triangle or square) $K$.
We show that this operator is stable with respect to polynomial degrees and
satisfies the commuting diagram property. We also establish an estimate for the
interpolation error in the norm of the space $\tilde\bH^{-1/2}(\div,K)$,
which is closely related to the energy spaces for boundary integral formulations
of time-harmonic problems of electromagnetics in three dimensions.


\end{abstract}

\section{Introduction and main results} \label{sec_intro}
\setcounter{equation}{0}

This paper addresses the problem of $\bH(\div)$-conforming interpolation
of low-regular vector fields
by high order polynomials.
Corresponding $p$-interpolation operators are relevant for the analysis of high order
{\em boundary element} approximations for time-harmonic problems of electromagnetics.

Aiming at high-order {\em finite element} (FE) approximations of Maxwell's equations,
Demkowicz and Babu{\v s}ka \cite{DemkowiczB_03_pIE}
introduced and analysed two projection-based $p$-inter\-po\-la\-ti\-on operators
satisfying the commuting diagram property (de Rham diagram).
These are the $H^1$-conforming interpolation operator
$\Pi^{1}_p:\, H^{1+r}(K) \rightarrow \CP_p(K)$ and
the $\bH(\curl)$-conforming interpolation operator
$\Pi^{\curl}_p:\, \bH^r(K) \cap \bH(\curl,K) \rightarrow \bCP^{\rm Ned}_p(K)$;
here $r > 0$ in both cases, $K$ is the reference element (either triangle or square),
$\CP_p(K)$ is the set of polynomials of degree $\le p$ on $K$,
and $\bCP^{\rm Ned}_{p}(K)$ is the $\bH(\curl)$-conforming (first) N{\' e}d{\' e}lec
space of degree $p$ (precise definitions of all involved Sobolev spaces and polynomial sets
are given in Section~\ref{sec_spaces} below).

In 2D, the operators curl and $\div$ are isomorphic.
The corresponding polynomial set isomorphic to the
N{\' e}d{\' e}lec space $\bCP^{\rm Ned}_{p}(K)$ is
the Raviart-Thomas (RT) space denoted by $\bCP^{\rm RT}_p(K)$.
Therefore, the results of \cite{DemkowiczB_03_pIE} related to the operator $\Pi^{\curl}_p$
can be used also in the $\bH(\div)$-conforming settings
(we will denote the corresponding $\bH(\div)$-conforming projection-based
interpolation operator by $\Pi^{\div}_p$). In particular,
given a vector field $\bu \in \bH^r(K) \cap \bH(\div,K)$ with $r > 0$,
the interpolant $\tilde\bu^p = \Pi^{\div}_p\,\bu \in \bCP^{\rm RT}_p(K)$
is defined as the sum of three terms:
\be \label{Old_E^p}
    \tilde\bu^p = \bu_1 + \bu^p_2 + \tilde\bu^p_3,
\ee
where $\bu_1$ is a lowest order interpolant, $\bu^p_2$ is the sum of edge
interpolants, and $\tilde\bu^p_3$ is the interior interpolant (a more detailed
description of these interpolants is given in Section~\ref{sec_inter_old}).
As follows from \cite{DemkowiczB_03_pIE}, the following diagram commutes:
\be \label{Old_deRham}
    \ba{ccccc}
    H^{1+r}(K)                & \stackrel{\bcurl}{\longrightarrow} &
    \bH^r(K) \cap \bH(\div,K) & \stackrel{\div}{\longrightarrow}   & L^2(K)
    \cr
    \quad\left\downarrow{\Large\strut}\right.\,\Pi^1_p         &
                                                               &
    \qquad\left\downarrow{\Large\strut}\right.\,\Pi^{\div}_p &
                                                               &
    \qquad\left\downarrow{\Large\strut}\right.\,\Pi^{0}_{p-1}
    \cr
    \CP_{p}(K)          & \stackrel{\bcurl}{\longrightarrow} &
    \bCP^{\rm RT}_p(K)  & \stackrel{\div}{\longrightarrow}   & \CP_{p-1}(K),
    \ea
\ee
where $\Pi^0_p:\, L^2(K) \rightarrow \CP_p(K)$ denotes the standard
$L^2$-projection onto the set of polynomials $\CP_p(K)$.

The commuting diagram property, and the corresponding $p$-interpolation error estimates, have
immediate applications to the analysis of high-order FE discretisations
of time-harmonic Maxwell's equations. In particular, these results are critical to
prove the discrete compactness property, which in turn implies the convergence of FE
approximations for Maxwell's equations, as well as for the error analysis
(see \cite{BoffiDC_03_DCp,BoffiCDD_06_Dhp,Hiptmair_DCp,BespalovH_OEE,BoffiCDDH_DCp}).
We note that classical N{\' e}d{\' e}lec or RT interpolation operators
(see, e.g., \cite{BrezziF_91_MHF}) are not suitable for these purposes,
as they are not stable (with respect to the polynomial degree $p$) for low-regular fields and
do not work equally well for triangular and parallelogram elements.

When time-harmonic problems of electromagnetics are posed in infinite domains
(e.g., outside a scatterer), it is convenient to reformulate them as
a boundary integral equation (on the surface of the scatterer).
The energy spaces for such boundary integral equations (BIE) involve
Sobolev spaces of negative order for both the vector field and its divergence (a typical example
is the space $\bH^{-1/2}(\div,\G)$ in the case of a smooth (closed) surface $\G$).
Then, it is common to use the $\bH(\div)$-conforming boundary elements (e.g., of RT type)
to discretise these BIE. The fundamental problem is that the underlying integral operator
is not coercive, and the convergence analysis of the boundary element methods (BEM)
requires a suitable regular decomposition of the energy space into the space
of divergence-free vector fields and the complementary space, cf. \cite{Buffa_05_RDS}.
In the case of Maxwell's source problem it is possible to use a decomposition, where
the complementary space is regular enough even on non-smooth surfaces.
Then, the $\bH(\div)$-conforming $p$-interpolation operator of Demkowicz and Babu{\v s}ka
is applicable for the convergence and error analysis
of the $p$- and the $hp$-BEM (see \cite{BespalovHH_Chp,BespalovH_hpA}).
However, when considering the boundary integral formulation for the Maxwell eigenvalue
problem, the orthogonal Hodge decomposition of the energy space
(see \cite{BuffaHvPS_03_BEM,BuffaC_03_EFI}) must be used to prove the discrete
compactness property. In this case, the regularity issues on non-smooth
surfaces affect the smoothness of the complementary space and prevent one
from using the known $\bH(\div)$-conforming interpolation operators.
Hence, the aim of this paper is to introduce and analyse a new
$\bH(\div)$-conforming $p$-interpolation operator, which is stable with respect to $p$ and
retains the commuting diagram property analogous to (\ref{Old_deRham}),
but assumes less regularity than $\Pi^{\div}_p$
(namely, $\bH^r(K) \cap \tilde\bH^{-1/2}(\div,K)$-regularity with $r > 0$).
This new interpolation operator will be denoted by $\Pi^{\div,-\frac 12}_p$.

Given a vector field
$\bu \in \bH^r(K) \cap \tilde\bH^{-1/2}(\div,K)$ with $r > 0$, we
define the interpolant $\bu^p = \Pi^{\div,-\frac 12}_p \bu \in \bCP^{\rm RT}_p(K)$
in a similar way as the interpolant
$\tilde\bu^p = \Pi^{\div}_p \bu \in \bCP^{\rm RT}_p(K)$ (see (\ref{Old_E^p})):
\be \label{E^p}
    \bu^p = \bu_1 + \bu^p_2 + \bu^p_3.
\ee
Here, $\bu_1$ and $\bu^p_2$ are exactly the same as for the interpolant
$\Pi^{\div}_p \bu$ (see (\ref{E_1}) and (\ref{E_2^p}), respectively),
whereas $\bu^p_3 \in \bCP^{\rm RT,0}_p(K)$ is determined by
solving the following system of equations:
\bea
     \<\div(\bu - (\bu_1 + \bu_2^p + \bu_3^p)),\div\,\bv\>_{\tH^{-1/2}(K)} = 0
     & \quad &
     \forall \bv \in \bCP^{\rm RT,0}_p(K),
     \label{E_3^p_1}
     \\[5pt]
     \<\bu - (\bu_1 + \bu_2^p + \bu_3^p),\bcurl\,\phi\>_{0,K} = 0
     & \quad &
     \forall \phi \in \CP^{0}_p(K),
     \label{E_3^p_2}
\eea
where $\<\cdot,\cdot\>_{\tH^{-1/2}(K)}$ and $\<\cdot,\cdot\>_{0,K}$ denote
the $\tH^{-1/2}(K)$- and the $\bL^2(K)$-inner products respectively.

It is easy to see that our construction of the interpolation operator $\Pi^{\div,-\frac 12}_p$
is much in the spirit of \cite{DemkowiczB_03_pIE}. However, we stress the use
of the $\tH^{-1/2}$-inner product in (\ref{E_3^p_1}), which is more
natural in the boundary element settings (this is in contrast to the $L^2$-inner
product employed in the definition of $\Pi^{\div}_p$, see (\ref{Old_E_3^p_1})).
Thus, the $\tH^{-1/2}(K)$-inner product has to be written in an appropriate explicit form.
Of particular importance for our analysis is the following property
of the $\tH^{-1/2}(K)$-inner product $\<u,v\>_{\tH^{-1/2}(K)}$:
for a constant function $v$, it reduces to the $L^2(K)$-inner product, i.e.,
\be \label{key_ip_property}
    \<u,1\>_{\tH^{-1/2}(K)} = \<u,1\>_{0,K} \qquad \forall\, u \in \tH^{-1/2}(K).
\ee
An inner product satisfying this property is presented in the Appendix
(see Lemma~\ref{lm_tilde-1/2-norm} and Lemma~\ref{lm_tilde_-1/2_ip}).

In the following three theorems we formulate the main results of the paper --
the properties of the operator $\Pi^{\div,-\frac 12}_p$
(all proofs are given in Section~\ref{sec_proofs} below).

The first theorem justifies the definition of the operator and
states its continuity.

\begin{theorem} \label{thm_stab}
For $r > 0$ the operator
\[
  \Pi_p^{\div,-\frac 12}:\;
  \bH^r(K) \cap \tilde\bH^{-1/2}(\div,K) \rightarrow
  \bL^2(K) \cap \tilde\bH^{-1/2}(\div,K)
\]
is well defined and bounded, with its operator norm being independent of $p$,
i.e., there exists a constant $C>0$ independent of $p$ (but depending on $r$)
such that
\be \label{stab}
    \Big\|\Pi_p^{\div,-\frac 12}\Big\|_{\CL} \le C,
\ee
where $\|\cdot\|_{\CL}$ is the operator norm in the space
$\CL\Big(\bH^r(K) \cap \tilde\bH^{-1/2}(\div,K),
 \bL^2(K) \cap \tilde\bH^{-1/2}(\div,K)\Big)$.
Moreover, the operator $\Pi_p^{\div,-\frac 12}$
preserves polynomial vector fields, i.e., $\Pi_p^{\div,-\frac 12} \bv_p = \bv_p$
for any $\bv_p \in \bCP^{\rm RT}_p(K)$.
\end{theorem}

The next theorem states the commuting diagram property analogous to (\ref{Old_deRham}).

\begin{theorem} \label{thm_deRham}
For $r > 0$ the following diagram commutes:
\be \label{deRham}
    \ba{ccccc}
    H^{1+r}(K)                             & \stackrel{\bcurl}{\longrightarrow} &
    \bH^r(K) \cap \tilde\bH^{-1/2}(\div,K) & \stackrel{\div}{\longrightarrow}   & \tH^{-1/2}(K)
    \cr
    \quad\left\downarrow{\Large\strut}\right.\,\Pi^1_p                 &
                                                                       &
    \qquad\left\downarrow{\Large\strut}\right.\,\Pi^{\div,-\frac 12}_p &
                                                                       &
    \qquad\left\downarrow{\Large\strut}\right.\,\Pi^{-1/2}_{p-1}
    \cr
    \CP_{p}(K)          & \stackrel{\bcurl}{\longrightarrow} &
    \bCP^{\rm RT}_p(K)  & \stackrel{\div}{\longrightarrow}   & \CP_{p-1}(K),
    \ea
\ee
where $\Pi^{-1/2}_p:\, \tH^{-1/2}(K) \rightarrow \CP_p(K)$ denotes the
$\tH^{-1/2}$-projector.
\end{theorem}

The third theorem provides an error estimate
for the interpolation operator $\Pi_p^{\div,-\frac 12}$
in the norm of the space $\tilde\bH^{-1/2}(\div,K)$, which is closely related
to the energy space for the electric field integral equation (a boundary integral
formulation of Maxwell's equations in 3D).

\begin{theorem} \label{thm_err_estimate}
If $\bu \in \bH^r(\div,K)$ with $r > 0$, then
there exists a positive constant $C$ independent of $\bu$ and $p$ such that
\be \label{err_estimate}
    \|\bu - \Pi^{\div,-\frac 12}_p\,\bu\|_{\tilde\bH^{-1/2}(\div,K)} \le
    C\,p^{-(r+1/2)}\,\|\bu\|_{\bH^r(\div,K)}.
\ee
\end{theorem}

\begin{remark} \label{rem_1}
Using similar constructions it is possible to introduce
a stable $\bH(\div)$-conforming $p$-interpolation operator for even less regular
vector fields $\bu \in \bH^r(K) \cap \tilde\bH^{s}(\div,K)$,
$r > 0$, $-1 \le s < -\frac 12$.
However, our proof of the commuting diagram property carries over to
this operator if the $\tH^{s}(K)$-inner product
reduces to the $L^2$-inner product for a constant function (cf. property
{\rm (\ref{key_ip_property})}), which is an open problem.
Although we note that such an operator would be of purely theoretical interest.
\end{remark}

\begin{remark} \label{rem_2}
Theorem~{\rm \ref{thm_err_estimate}} states the interpolation error estimate
for sufficiently regular vector fields, for which one can also
apply the operator $\Pi^{\div}_p$.
For BIE of electromagnetics on open surfaces, the solution is less regular
and belongs to $\bH^r(\div,\G)$, where $r \in (-\frac 12, 0)$ and $\G$ is an open
surface (see {\rm \cite[Section 4.4]{CostabelD_00_SEF}} and
{\rm \cite[Appendix~A]{BespalovH_NpB}}).
To obtain the error estimate for the corresponding BEM in this case, one can apply the (global)
orthogonal projection $P_{p}$ with respect to the energy norm $\|\cdot\|_{\bX}$.
Then the estimate for $\|\bu - P_{p}\bu\|_{\bX}$ can be reduced
to the estimate obtained in Theorem~{\rm \ref{thm_err_estimate}}
in the same way as in {\rm \cite{BuffaC_03_EFI}} or {\rm \cite{BespalovH_hpA}}.
Note that using the projector $P_{p}$ locally does not guarantee the conformity
of approximations (i.e., the continuity of normal components across interelement boundaries).
Moreover, the projector $P_{p}$ does not satisfy the commuting diagram property
in {\rm (\ref{deRham})}, and, thus, it is not suitable for such purposes as the
convergence analysis and the proof of the discrete compactness.
\end{remark}

The rest of the paper is organised as follows.
Section~\ref{sec_prelim} gives necessary preliminaries:
we introduce the notation, recall definitions of functional
spaces of scalar functions and vector fields, and collect
auxiliary results. In particular, we give a more detailed description
of the interpolation operators $\Pi^{1}_p$, $\Pi^{\div}_p$ and summarise their properties
(see \S\ref{sec_inter_old}). In Section~\ref{sec_proofs} we prove the main
theorems formulated above. Finally, in the Appendix we introduce some equivalent norms
in the Sobolev spaces $H^r(K)$ and $\tH^r(K)$ ($r = \pm\frac 12$),
derive expressions for corresponding inner products, and establish the key property
(\ref{key_ip_property}) for the $\tH^{-1/2}$-inner product.

\section{Preliminaries} \label{sec_prelim}
\setcounter{equation}{0}

\subsection{Functional spaces and polynomial sets} \label{sec_spaces}

In what follows, $p \ge 0$ will always specify a polynomial degree and $C$ denotes
a generic positive constant which is independent
of $p$ and involved functions, unless stated otherwise.
Furthermore, throughout the paper, $K$ is either the equilateral reference triangle
$T = \{\bx = (x_1,x_2);\; x_2 > 0,\ x_2 < x_1 \sqrt 3,\ x_2 < (1-x_1)\sqrt 3\}$ or
the reference square $Q=(0,1)^2$.
A generic edge of $K$ will be denoted by $\ell$,
and $\bn$ denotes the outward normal unit vector to $\partial K$.

We will use the standard definitions for the Sobolev spaces $H^r(\Omega)$ ($r \ge 0$)
of scalar functions on $\Omega$, see, e.g., \cite{LionsMagenes}
(hereafter, $\Omega$ is either the unit interval $I = (0,1)$
or the reference element $K$). The norms in these spaces are denoted by
$\|\cdot\|_{H^r(\Omega)}$.
For $r \in (0,1)$ we will also need the Sobolev spaces $\tH^r(\Omega)$
which are defined by interpolation.
We use the real K-method of interpolation (see \cite{LionsMagenes}) to define
\[
   \tilde H^{r}(\Omega) = \Big(L^2(\Omega), H_0^t(\Omega)\Big)_{\frac rt,2}
   \quad (1/2 < t \le 1,\ 0<r<t).
\]
Here, $H_0^t(\Omega)$  ($0<t\le 1$) is the completion of $C_0^\infty(\Omega)$ in
$H^t(\Omega)$ and we identify $H_0^1(\Omega)$ with $\tilde H^1(\Omega)$.
Note that the Sobolev spaces $H^r(\Omega)$ also satisfy the interpolation property
\[
   H^r(\Omega) = \Big(L^2(\Omega), H^1(\Omega)\Big)_{r,2}\quad (0<r<1)
\]
with equivalent norms.

The $L^2$-inner product and the corresponding $L^2$-norm on $\Omega$ are
denoted by $\<\cdot,\cdot\>_{0,\Omega}$ and $\|\cdot\|_{0,\Omega}$, respectively.
For $r\in[-1,0)$ the Sobolev spaces and their norms are defined by duality
with $L^2(\Omega) = H^0(\Omega) = \tH^0(\Omega)$ as pivot space:
\[
   H^r(\Omega) = \big(\tilde H^{-r}(\Omega)\big)',\quad
   \tilde H^r(\Omega) = \big(H^{-r}(\Omega)\big)',
\]
\be \label{dual_norm}
    \|u\|_{H^{r}(\Omega)} = \sup_{0\not=v\in\tH^{-r}(\Omega)}
    {|\<u,v\>_{0,\Omega}| \over{\|v\|_{\tH^{-r}(\Omega)}}},
    \quad
    \|u\|_{\tH^{r}(\Omega)} = \sup_{0\not=v\in H^{-r}(\Omega)}
    {|\<u,v\>_{0,\Omega}| \over{\|v\|_{H^{-r}(\Omega)}}}.
\ee
Note that the Sobolev spaces $H^r$ and $\tH^r$
on any edge $\ell \subset \partial K$ are defined by using the
definitions of the corresponding spaces on the interval $I$.

In the Appendix we consider some other expressions for norms
in the Sobolev spaces $H^r(K)$ and $\tH^r(K)$ with $r = \pm\frac 12$.
We will prove their equivalence to the norms defined above, and
we will also derive expressions for corresponding inner products.

Throughout the paper, we use boldface symbols for vector fields.
The spaces (or sets) of vector fields are denoted in boldface as well
(e.g., $\bH^r(K) = (H^r(K))^2$), with their norms and inner
products being defined component-wise.
Similarly to the scalar case, the norm and inner product in $\bL^2(K)$
will be denoted by $\<\cdot,\cdot\>_{0,K}$ and $\|\cdot\|_{0,K}$, respectively,
which should not lead to any confusion.
The standard notation will be used for differential operators
$\grad = (\partial/\partial x_1,\,\partial/\partial x_2)$,
$\div = \grad\,\cdot$, $\curl = \grad\mbox{\small $\times$}$, and
for the Laplace operator $\Delta = \div\,\grad$.

Furthermore, we will use the following spaces
\[
  \bH^r(\div,K) :=
  \{\bu \in \bH^{r}(K);\;
  \div\,\bu \in H^{r}(K)\},\quad r \ge 0
\]
and
\[
  \tilde\bH^{r}(\div,K) :=
  \{\bu \in \tilde\bH^{r}(K);\;
  \div\,\bu \in \tH^{r}(K)\}, \quad r \in [-1,-\hbox{$\frac 12$}].
\]
These spaces are equipped with their graph norms
$\|\cdot\|_{\bH^r(\div,K)}$ and $\|\cdot\|_{\tilde\bH^{r}(\div,K)}$,
respectively.
For $r=0$ we drop the superscript in the above notation:
$\bH^0(\div,K) = \bH(\div,K)$.

Finally, we will need two sub-spaces incorporating homogeneous
boundary conditions for the trace of the normal component on $\partial K$.
By $\bH_0(\div,K)$ (resp., $\tilde\bH^{-1/2}_0(\div,K)$) we denote the
subspace of elements $\bu \in \bH(\div,K)$ (resp., $\bu \in \tilde\bH^{-1/2}(\div,K)$)
such that for all $v \in C^{\infty}(K)$ there holds
\be \label{X_property}
    \<\bu,\grad v\>_{0,K} + \<\div\,\bu,v\>_{0,K} = 0.
\ee
We note that if $\bu \in \tilde\bH^{-1/2}_0(\div,K)$,
then identity (\ref{X_property}) holds for any $v\in H^{3/2}(K)$ by density.
In particular, $\tilde\bH^{-1/2}_0(\div,K)$ is a closed subspace of
$\tilde\bH^{-1/2}(\div,K)$.

Let us now introduce the polynomial sets we need.
By $\CP_p(I)$ we denote the set of polynomials of degree $\le p$ on
the interval $I$, and $\CP^0_p(I)$ denotes the subset of $\CP_p(I)$
which consists of polynomials vanishing at the end points of $I$.
In particular, these two sets will be used for any edge $\ell \subset \partial K$.

Further, $\CP^1_p(T)$ denotes the set of polynomials on $T$ of total degree $\le p$, and
$\CP^2_{p_1,p_2}(Q)$ is the set of polynomials on $Q$ of degree $\le p_1$ in $x_1$
and of degree $\le p_2$ in $x_2$. For $p_1 = p_2 = p$ we denote
$\CP^2_{p}(Q) = \CP^2_{p,p}(Q)$, and we will use the unified notation
$\CP_{p}(K)$, which refers to $\CP^1_{p}(T)$ if $K=T$ and to $\CP^2_{p}(Q)$ if $K=Q$.
The corresponding set of polynomial (scalar) bubble functions on $K$ is denoted
by $\CP^0_{p}(K)$.

Let us denote by $\bCP^{\rm RT}_p(K)$ the RT-space of order $p\ge 1$ on
the reference element $K$ (see, e.g., \cite{BrezziF_91_MHF, RobertsT_91_MHM}), i.e.,
\[
  \bCP^{\rm RT}_p(K) =
  (\CP_{p-1}(K))^2 \oplus \bx \CP_{p-1}(K) =
  \cases{
         (\CP^1_{p-1}(T))^2 \oplus \bx \CP^1_{p-1}(T)
         & \hbox{if \ $K = T$},
         \cr
         \noalign{\vspace{5pt}}
         \CP^2_{p,p-1}(Q) \times \CP^2_{p-1,p}(Q)
         & \hbox{if \ $K = Q$}.
         \cr
        }
\]
The subset of $\bCP^{\rm RT}_p(K)$ which consists of vector-valued polynomials
with vanishing normal trace on the boundary $\partial K$
(vector bubble-functions) will be denoted by $\bCP^{\rm RT,0}_p(K)$.

\subsection{Auxiliary lemmas} \label{sec_aux}

First, let us formulate the following result, which will be used frequently in
what follows.

\begin{lemma} \label{lm_trace}
The normal trace mapping $\bu \mapsto \bu \cdot \bn$ defines
a linear and continuous operator from $\bH^s(K) \cap \tilde\bH^{-1+s}(\div,K)$
to $H^{-1/2+s}(\partial K)$ for $s \in [0,\frac 12)$.
\end{lemma}

\begin{proof}
Let us denote by $\g_{\rm tr}$ the standard (scalar) trace operator with
$\g_{\rm tr}: H^{1-s}(K) \rightarrow H^{1/2-s}(\partial K)$ for $s \in [0,\frac 12)$, and
let $\g_{\rm tr}^{-1}: H^{1/2-s}(\partial K) \rightarrow H^{1-s}(K)$ be
a right inverse of $\g_{\rm tr}$. Let $\bu \in \bH^s(K) \cap \tilde\bH^{-1+s}(\div,K)$.
Taking an arbitrary $v \in H^{1/2-s}(\partial K)$ we integrate by parts to obtain
\beas
     \int\limits_{\partial K} (\bu \cdot \bn)\,v\,d\sigma
     & = &
     \int\limits_{K} (\div\,\bu)\,\g_{\rm tr}^{-1} v\,d\bx +
     \int\limits_{K} \bu \cdot \grad(\g_{\rm tr}^{-1} v)\,d\bx
     \\[3pt]
     & \le &
     \|\div\,\bu\|_{\tH^{-1+s}(K)}\,\|\g_{\rm tr}^{-1} v\|_{H^{1-s}(K)} +
     \|\bu\|_{\bH^s(K)}\,\|\grad(\g_{\rm tr}^{-1} v)\|_{\bH^{-s}(K)}
     \\[3pt]
     & \le &
     C \left(
             \|\bu\|_{\bH^s(K)} + \|\div\,\bu\|_{\tH^{-1+s}(K)}
      \right)
     \|v\|_{H^{1/2-s}(\partial K)}.
\eeas
Hence, $\bu \cdot \bn \in H^{-1/2+s}(\partial K)$ and we prove the continuity
of the normal trace mapping:
\beas
     \|\bu \cdot \bn\|_{H^{-1/2+s}(\partial K)}
     & = &
     \sup_{0\not=v\in H^{1/2-s}(\partial K)}
     {|\int_{\partial K} (\bu \cdot \bn)\, v\, d\sigma|
     \over{\|v\|_{H^{1/2-s}(\partial K)}}}
     \\[3pt]
     & \le &
     C \left(
             \|\bu\|_{\bH^s(K)} + \|\div\,\bu\|_{\tH^{-1+s}(K)}
      \right).
\eeas
\end{proof}

We will also need the following $p$-approximation result
in 2D (see \cite[Lemma~4.1]{BabuskaS_87_hpF}).

\begin{lemma} \label{lm_2D_p-approx}
Let $K$ be the reference triangle or square.
Then there exists a family of operators
$\{\pi_p\},\ p=1,2,\ldots,\ \pi_p:\, H^r(K)\rightarrow \CP_p(K)$
such that for any $f \in H^r(K)$, $r \ge 0$ there holds
\[
  \|f - \pi_p f\|_{H^t(K)} \le
  C p^{-(r-t)} \|f\|_{H^r(K)},\qquad  0\le t\le r.
\]
Moreover, $\pi_p$ preserves polynomials of degree $p$, i.e.,
$\pi_p f = f$ if $f \in \CP_p(K)$.
\end{lemma}

We use this result, in particular, to prove the next lemma,
which provides an optimal error estimate for the $\tH^{-1/2}$-projector
$\Pi^{-1/2}_p:\, \tH^{-1/2}(K) \rightarrow \CP_p(K)$.

\begin{lemma} \label{lm_tH^{-1/2}_estimate}
Let $\phi \in H^r(K)$, $r > -\frac 12$.
Then for any $p \ge 0$ there holds
\be \label{tH^{-1/2}_estimate}
    \|\phi - \Pi^{-1/2}_p\,\phi\|_{\tH^{-1/2}(K)} \le
    C (p+1)^{-(1/2+r)} \|\phi\|_{H^r(K)}.
\ee
\end{lemma}

\begin{proof}
If $p = 0$ then (\ref{tH^{-1/2}_estimate}) is trivial. Let $p \ge 1$.
First, we assume that $r > 0$.
Using the standard duality argument and the $p$-approximation result of
Lemma~\ref{lm_2D_p-approx}, we estimate the error of the $L^2$-projection
$\Pi^0_p:\, L^2(K) \rightarrow \CP_p(K)$ in the $\tH^{-1/2}$-norm:
\beas
     \|\phi - \Pi_p^0\phi\|_{\tH^{-1/2}(K)}
     & \le &
     \|\phi - \Pi_p^0\phi\|_{0,K}
     \sup_{\varphi \in H^{1/2}(K)\setminus\{0\}}\;
     \inf_{\varphi_p \in \CP_p(K)}
     \frac {\|\varphi - \varphi_p\|_{0,K}}{\|\varphi\|_{H^{1/2}(K)}}
     \\
     & \le &
     \|\phi - \Pi_p^0\phi\|_{0,K}
     \sup_{\varphi \in H^{1/2}(K)\setminus\{0\}}
     \frac {\|\varphi - \Pi_p^0\varphi\|_{0,K}}{\|\varphi\|_{H^{1/2}(K)}}
     \\
     & \le &
     C\,(p+1)^{-(1/2+r)}\,\|\phi\|_{H^r(K)}.
\eeas
This estimate yields (\ref{tH^{-1/2}_estimate}) due to the minimization property
of the $\tH^{-1/2}$-projection.

Now, let $r \in (-\frac 12,0]$. Assuming that $\phi \in H^s(K) = \tH^s(K)$ with some
$s \in (0,\frac 12)$ and using the first part of the proof, we have
\[
  \|\phi - \Pi_p^{-1/2}\phi\|_{\tH^{-1/2}(K)} \le
  C\,(p+1)^{-(1/2+s)}\,\|\phi\|_{H^s(K)}.
\]
On the other hand, it is trivial that
\[
  \|\phi - \Pi_p^{-1/2}\phi\|_{\tH^{-1/2}(K)} \le
  \|\phi\|_{\tH^{-1/2}(K)}.
\]
Therefore, we prove by interpolation that
\beas
     \|\phi - \Pi_p^{-1/2}\phi\|_{\tH^{-1/2}(K)}
     & \le &
     C\,(p+1)^{-(1/2+r)}\,\|\phi\|_{\tH^r(K)}
     \cr\cr
     & \le &
     C(r)\,(p+1)^{-(1/2+r)}\,\|\phi\|_{H^r(K)}\qquad
     \forall\,\phi \in H^s(K).
\eeas
Hence, by density of regular functions in $H^r(K)$, we obtain (\ref{tH^{-1/2}_estimate}),
and the proof is finished.
\end{proof}

The following lemma states the inverse inequality for polynomials on
the reference element~$K$.

\begin{lemma} \label{lm_inverse}
Let $v_p \in \CP_p(K)$. Then for any $s,\,r \in [-1,1]$ with
$s \le r$ there holds
\[
  \|v\|_{H^r(K)} \le C\, p^{2(r-s)}\, \|v\|_{H^s(K)},
\]
where $C$ is a positive constant independent of $p$.
\end{lemma}

For $r \ge 0$, $s = 0$ the proof is based on Schmidt's inequality and
given in \cite{Dorr_84_ATp} for both types of reference elements
(see Lemma~5.1 and its proof therein). By using interpolation arguments
and induction, this result has been
extended in \cite{Heuer_01_ApS} to the full range of parameters $s,\,r \in [-1,1]$.

\subsection{The regularized Poincar{\' e} integral operators} \label{sec_Poincare}

In \cite{CostabelM_BRP}, Costabel and McIntosh studied a regularized version of the
Poincar{\' e}-type integral operator acting on differential forms in ${\field{R}}^n$.
They proved, in particular, that this operator is bounded on a wide range of
functional spaces including the whole scale of Sobolev spaces $H^r(\Omega)$
($r \in {\field{R}}$) on a bounded Lipschitz domain $\Omega$ which is starlike
with respect to an open ball. Moreover, the essential polynomial preserving
property of the classical Poincar{\' e} map is retained by its regularized version.
Thus, the results of \cite{CostabelM_BRP} have immediate applications to the analysis
of high-order elements (see, e.g., \cite{Hiptmair_DCp,BespalovH_OEE,BespalovHH_Chp}).

Let us formulate some results of \cite{CostabelM_BRP} in two particular cases.
Namely, we will define two Poincar{\' e}-type integral operators: one operator
acts on scalar functions, and the other one acts on divergence-free vector fields.
In both cases the functions and vector fields are defined on the reference
element $K$. Denoting by $B$ an open ball in $K$,
let us consider a smoothing function
\[
  \theta \in C^{\infty}({\field{R}}^2),\quad
  \supp\,\theta \subset B,\quad
  \int\limits_{B} \theta(\bolda)\,d\bolda = 1,\quad
  \bolda = (a_1,a_2).
\]
Then the first regularized Poincar{\' e}-type integral operator
$R:\,C^{\infty}(\bar K) \rightarrow (C^{\infty}(\bar K))^2$
(i.e., the operator acting on scalar functions) is defined as
$R\psi = (R_1,R_2)$, where
\[
  R_i(\bx) := \int\limits_{B} \theta(\bolda)\,(x_i - a_i)
              \int\limits_0^1 t \psi(\bolda + t(\bx - \bolda))\,dt\,d\bolda,\quad
  i = 1,2.
\]
The second operator acting on vector fields is defined as follows:
\[
  \ba{l}
  A:\,(C^{\infty}(\bar K))^2 \rightarrow C^{\infty}(\bar K),
  \\
  \displaystyle{
  A\bu(\bx) :=
  \int\limits_{B} \theta(\bolda)
  \bigg(
        (x_2 - a_2) \int\limits_0^1 u_1(\bolda + t(\bx - \bolda))\,dt -
        (x_1 - a_1) \int\limits_0^1 u_2(\bolda + t(\bx - \bolda))\,dt
  \bigg)
  d\bolda,
  }
  \ea
\]
where $\bu = (u_1,u_2)$.

The following properties of the operators $R$ and $A$ are easy to check directly
(see also \cite[Proposition~4.2]{CostabelM_BRP}):

\begin{itemize}

\item[(R1)]
$R$ is a right inverse of the $\div$ operator, i.e.,
\[
  \div(R\psi) = \psi\qquad \forall\,\psi \in H^r(K),\quad r \ge 0;
\]

\item[(A1)]
if $\bu$ is divergence-free, then $A$ is a right inverse of the vector curl, i.e.,
\[
  \bcurl(A\bu) = \bu\qquad
  \forall\,\bu \in \bH^r(\div 0, K) =
  \{\bu \in \bH^r(K);\; \div\,\bu = 0\ \hbox{in $K$}\},\quad
  r \ge 0.
\]

\end{itemize}

The operators $R$ and $A$ satisfy the following continuity properties
(see \cite[Corollary~3.4]{CostabelM_BRP}):

\begin{itemize}

\item[(R2)]
the mapping $R$ defines a bounded operator $H^{r-1}(K) \rightarrow \bH^{r}(K)$
for any $r \ge 0$;

\item[(A2)]
the mapping $A$ defines a bounded operator $\bH^r(K) \rightarrow H^{r+1}(K)$
for any $r \ge 0$.

\end{itemize}

Furthermore, the operators $R$ and $A$ preserve polynomials:

\begin{itemize}

\item[(R3)]
$R$ maps $\CP_p(K)$ into $\bCP^{\rm RT}_{p+1}(K)$;

\item[(A3)]
$A$ maps $\bCP^{\rm RT}_{p}(K)$ into $\CP_{p}(K)$.

\end{itemize}

We will use the operators $R$ and $A$ to prove the following auxiliary lemma.

\begin{lemma} \label{lm_Poincare}
Let $r > 0$ and $s \ge r - 1$. If $\bu \in \bH^r(K)$ and $\div\,\bu \in H^s(K)$,
then there exist a function $\psi \in H^{r+1}(K)$
and a vector field $\bv \in \bH^{s+1}(K)$ such that
\be \label{P_1}
    \bu = \bcurl\,\psi + \bv.
\ee
Moreover,
\be \label{P_2}
    \|\bv\|_{\bH^{s+1}(K)} \le C\,\|\div\,\bu\|_{H^s(K)}\qquad
    \hbox{and}\qquad
    \|\psi\|_{H^{r+1}(K)} \le C\,\|\bu\|_{\bH^r(K)}.
\ee
\end{lemma}

\begin{proof}
The proof is exactly the same as for Lemma~2.3 in \cite{BespalovH_OEE}.
We use the operators $R$ and $A$ to define $\psi$ and $\bv$:
\[
  \bv := R(\div\,\bu) \in \bH^{s+1}(K), \quad
  \psi := A(\bu - R(\div\,\bu)) \in H^{\min\{r,s+1\}+1}(K) = H^{r+1}(K).
\]
Hence, due to properties (R1) and (A1), the vector field $\bu$ can be decomposed
as in (\ref{P_1}):
\[
  \bu = (\bu - R(\div\,\bu)) + R(\div\,\bu) = \bcurl\,\psi + \bv.
\]
Inequalities (\ref{P_2}) are then obtained by using the continuity properties
of the Poincar{\' e}-type operators and the boundedness of the divergence operator
as a mapping $\bH^r(K) \rightarrow H^{r-1}(K)$ for $r \ge 0$.
(cf.~\cite[Lemma~2.3]{BespalovH_OEE}).
\end{proof}

\begin{remark} \label{rem_Poincare}
Note that $\bu \in \bH^r(K)$ implies that $\div\,\bu \in H^{r-1}(K)$ for $r > 0$.
That is why, it is assumed in Lemma~{\rm \ref{lm_Poincare}} that $s \ge r - 1$.
\end{remark}

\subsection{Discrete Friedrichs inequalities} \label{sec_Friedrichs}

In this subsection we improve the discrete Friedrichs inequalities of
\cite[Theorem~1]{DemkowiczB_03_pIE}. This improvement has also become possible
due to the properties of the regularized Poincar{\' e} integral operators
which were established in~\cite{CostabelM_BRP} and summarised in the previous subsection.

\begin{lemma} \label{lm_Friedrichs}
There exist positive constants $C_1,\,C_2$ independent of $p$ such that
\begin{enumerate}
\item[{\rm (i)}]
for any $\bu \in \bCP^{\rm RT,0}_p(K)$ satisfying
$\<\bu,\bcurl\,\varphi\>_{0,K} = 0$ for all $\varphi \in \CP^0_p(K)$,
there holds
\be \label{Friedrichs_3}
    \|\bu\|_{0,K} \le C_1\,\|\div\,\bu\|_{\tH^{-1}(K)};
\ee
\item[{\rm (ii)}]
for any $\bu \in \bCP^{\rm RT}_p(K)$ satisfying
$\<\bu,\bcurl\,\varphi\>_{0,K} = 0$ for all $\varphi \in \CP_p(K)$,
there holds
\be \label{Friedrichs_4}
    \|\bu\|_{0,K} \le C_1\,\|\div\,\bu\|_{H^{-1}(K)}.
\ee
\end{enumerate}
\end{lemma}

\begin{proof}
Following the idea of \cite[Theorem~1]{DemkowiczB_03_pIE} the proof reduces to finding
a continuous right inverse of the divergence operator within appropriate polynomial
spaces. In particular, in order to prove the first statement of the lemma
one needs to construct an operator $\CT$ mapping
$\mathring{\CP}_{p-1}(K) := \{\psi \in \CP_{p-1}(K);\; \int_K \psi\,d\bx = 0\}$
into $\bCP^{\rm RT,0}_p(K)$ and satisfying the following properties:
\be \label{Friedrichs_5}
    \div\,(\CT\psi) = \psi \qquad \forall\,\psi \in \mathring{\CP}_{p-1}(K),
\ee
\be \label{Friedrichs_6}
    \|\CT\psi\|_{0,K} \le C\,\|\psi\|_{\tH^{-1}(K)} \qquad
    \forall\,\psi \in \mathring{\CP}_{p-1}(K).
\ee
Then, given any $\bu \in \bCP^{\rm RT,0}_p(K)$ such that
$\<\bu,\bcurl\,\varphi\>_{0,K} = 0$ for all $\varphi \in \CP^0_p(K)$,
we prove (\ref{Friedrichs_3}):
\beas
     \|\bu\|_{0,K}
     & = &
     \min_{\varphi \in \CP^0_p(K)} \|\bu - \bcurl\,\varphi\|_{0,K} \le
     \|\bu - (\bu - \CT(\div\,\bu))\|_{0,K}
     \cr\cr
     & = &
     \|\CT(\div\,\bu)\|_{0,K} \stackrel{(\ref{Friedrichs_6})}{\le}
     C\,\|\div\,\bu\|_{\tH^{-1}(K)}.
\eeas
Here, $\div\,\bu \in \mathring{\CP}_{p-1}(K)$ and the existence of
$\varphi \in \CP^0_p(K)$ satisfying $\bcurl\,\varphi = \bu - \CT(\div\,\bu)$ follows
from two facts:
\[
  \bu - \CT(\div\,\bu) \in \bCP^{\rm RT,0}_p(K)
\]
and
\[
  \div(\bu - \CT(\div\,\bu)) \stackrel{(\ref{Friedrichs_5})}{=}
  \div\,\bu - \div\,\bu = 0.
\]
Let us construct the operator $\CT$ satisfying (\ref{Friedrichs_5}), (\ref{Friedrichs_6}).
Let $\psi \in \mathring{\CP}_{p-1}(K)$.
Applying the regularized Poincar{\' e} operator $R$ we define
$\bv := R\psi$. Then $\bv \in \bCP^{\rm RT}_p(K)$, due to property (R3) of this operator.
Moreover, using property (R1) and the fact that $\int_K \psi\, d\bx = 0$
we conclude that $\bv\cdot\bn$ has zero average along $\partial K$:
\[
  \int\limits_{\partial K} \bv \cdot \bn\,d\sigma =
  \int\limits_{K} \div\,\bv\,d\bx =
  \int\limits_{K} \div(R\psi)\,d\bx =
  \int\limits_{K} \psi\,d\bx = 0.
\]
Hence, there exists a continuous
piecewise polynomial $\phi$ defined on $\partial K$ such that
$\phi|_{\ell} \in \CP_p(\ell)$ for any edge $\ell \subset \partial K$ and
$\frac{\partial\phi}{\partial\sigma} = \bv\cdot\bn$
on $\partial K$. Therefore, applying the polynomial extension result of
Babu{\v s}ka {\em et al.} \cite{BabuskaCMP_91_EPp}, we find a polynomial
$\tilde\phi \in \CP_p(K)$ such that $\tilde\phi|_{\partial K} = \phi$
and there holds
\[
    \ba{lcll}
    \|\bcurl\,\tilde\phi\|_{0,K}
    & \le &
    \|\tilde\phi\|_{H^1(K)} \le
    C\,\|\phi\|_{H^{1/2}(\partial K)/{\field{R}}}
    &
    \\[3ex]
    & \le &
    C\,\Big\|{\partial\phi\over{\partial\sigma}}\Big\|_{H^{-1/2}(\partial K)}
    & \mbox{(\cite[Lemma~2]{DemkowiczB_03_pIE})}
    \\[3ex]
    & = &
    C\,\|\bv\cdot\bn\|_{H^{-1/2}(\partial K)}
    & \mbox{(${\partial\phi\over{\partial\sigma}} = \bv\cdot\bn$)}
    \\[3ex]
    & \le &
    C\,(\|\bv\|_{0,K} + \|\div\,\bv\|_{\tH^{-1}(K)})
    & \mbox{(Lemma~\ref{lm_trace} with $s = 0$)}
    \\[3ex]
    & = &
    C\,(\|R\psi\|_{0,K} + \|\div(R\psi)\|_{\tH^{-1}(K)})\qquad
    & \mbox{($\bv = R\psi$).}
    \ea
\]
Hence, using properties (R1) and (R2) of the operator $R$, we obtain
\be \label{Friedrichs_7}
    \|\bcurl\,\tilde\phi\|_{0,K} \le C\,\|\psi\|_{\tH^{-1}(K)}.
\ee
Now we can define the desired operator $\CT$ as
$\CT\psi = R\psi - \bcurl\,\tilde\phi$. It is easy to check that
$\CT:\,\mathring{\CP}_{p-1}(K) \rightarrow \bCP^{\rm RT,0}_p(K)$ and
(\ref{Friedrichs_5}) holds. Making use of (\ref{Friedrichs_7})
and the continuity of the operator
$R:\,H^{-1}(K) \rightarrow \bL^2(K)$ (see (R2)), we also prove (\ref{Friedrichs_6}).

The proof of statement (ii) is analogous. In this case we can use the operator
$R{:}\,H^{-1}(K) \rightarrow \bL^2(K)$ for the desired continuous right inverse of $\div$.
Then $R \equiv \CT$ maps $\CP_{p-1}(K)$ into $\bCP^{\rm RT}_p(K)$ and (\ref{Friedrichs_4})
is derived similarly as above.
\end{proof}

\subsection{Existing $H^1$- and $\bH(\div)$-conforming interpolation operators} \label{sec_inter_old}

Let us briefly sketch the definitions and summarise the properties
of the $H^1$-conforming interpolation operator $\Pi^{1}_p$
and the $\bH(\div)$-conforming interpolation operator $\Pi^{\div}_p$
(see \cite{DemkowiczB_03_pIE} for details).

Let $g \in H^{1+r}(K)$, $r > 0$. To define the interpolant $\Pi^1_p\, g$, one starts
with the standard linear interpolation of $g$ at the vertices of $K$:
\[
  g_1 \in \CP_1(K),\quad g_1 = g \quad \hbox{at each vertex of $K$}.
\]
Then, for each edge $\ell \subset \partial K$, we define a polynomial
$g_{2,\ell}$ by using the projection
\be \label{int_1}
    g_{2,\ell} \in \CP_p^0(\ell):\quad
    \|(g - g_1)|_{\ell} - g_{2,\ell}\|_{\tH^{1/2}(\ell)} \rightarrow \min.
\ee
Extending $g_{2,\ell}$ by zero onto the remaining part of $\partial K$ (and keeping its
notation), using some polynomial extension $\CE_p$ from the boundary,
and summing up over all edges we define
\be \label{int_2}
    g_2^p := \sum\limits_{\ell \subset \partial K} \CE_p(g_{2,\ell}) \in \CP_p(K).
\ee
Finally, we define the polynomial bubble $g_3^p$ by projection in the $H^1$-semi-norm
\be \label{int_3}
    g_{3}^p \in \CP_p^0(K):\quad
    |g - (g_1 + g_{2}^p + g_3^p)|_{H^1(K)} \rightarrow \min.
\ee
Then the interpolant $\Pi^1_p\, g$ is defined as the sum
\be \label{int_4}
    \Pi^1_p\, g := g_1 + g_2^p +g_3^p \in \CP_p(K).
\ee

Now we proceed to the $\bH(\div)$-conforming interpolation operator.
Given a vector field $\bu \in \bH^r(K) \cap \bH(\div,K)$ with $r > 0$,
the interpolant $\tilde\bu^p = \Pi^{\div}_p\,\bu \in \bCP^{\rm RT}_p(K)$
is also defined as the sum of three terms:
\[  
    \tilde\bu^p = \bu_1 + \bu^p_2 + \tilde\bu^p_3.
\]  
Here, $\bu_1$ is a lowest order interpolant defined as
\be \label{E_1}
    \bu_1 = \sum\limits_{\ell \subset \partial K}
    \Big(\int\limits_{\ell} \bu\cdot\bn\,d\sigma\Big)\,\bphi_\ell\;,
\ee
where $\bphi_\ell$ are the standard basis functions
(associated with edges $\ell$) for $\bCP_1^{\rm RT}(K)$ such that
\[
  \bphi_\ell\cdot\bn =
  \cases{
         1 & \hbox{on $\ell$},\cr
         0 & \hbox{on $\partial K \backslash \ell$}.\cr
        }
\]
For any edge $\ell \subset \partial K$ one has
\be \label{E-E_1}
    \int\limits_{\ell} (\bu - \bu_1) \cdot \bn\,d\sigma = 0.
\ee
Hence, there exists a function $\psi$, defined on the boundary $\partial K$,
such that
\be \label{psi}
    {\partial\psi\over{\partial \sigma}} =  (\bu - \bu_1) \cdot \bn,\quad
    \psi = 0 \ \ \hbox{at all vertices}.
\ee
Then, for each edge $\ell$, we define $\psi_2^{\ell} \in \CP^0_p(\ell)$ by projection
\be \label{psi_2^l}
    \<\psi|_{\ell} - \psi_2^{\ell}, \phi\>_{\tH^{1/2}(\ell)} = 0\quad
    \forall \phi \in \CP^0_p(\ell)
\ee
(see Remark~\ref{rem_edge-norms} for the expression of
$\<\cdot,\cdot\>_{\tH^{1/2}(\ell)}$).
Extending $\psi_2^{\ell}$ by zero from $\ell$ onto $\partial K$ (and keeping
its notation), we denote by $\psi_{2,p}^{\ell} \in \CP_p(K)$
a polynomial extension of $\psi_2^{\ell}$ from $\partial K$ onto $K$, i.e.,
\be \label{psi_2,p^l}
    \psi_{2,p}^{\ell} \in \CP_p(K),\quad
    \psi_{2,p}^{\ell}|_{\ell} = \psi_{2}^{\ell},\quad
    \psi_{2,p}^{\ell}|_{\partial K\backslash\ell} = 0.
\ee
Then we set
\be \label{E_2^p}
    \bu_2^p = \sum\limits_{\ell \subset \partial K} \bu^p_{2,\ell},\ \ 
    \hbox{where \ } \bu^p_{2,\ell} = \bcurl\, \psi_{2,p}^{\ell}.
\ee
The interior interpolant $\tilde\bu^p_3$ is a vector bubble function living
in $\bCP^{\rm RT,0}_p(K)$ and satisfying the following system of equations:
\bea
     \<\div(\bu - (\bu_1 + \bu_2^p + \tilde\bu_3^p)),\div\,\bv\>_{0,K} = 0
     & \quad &
     \forall \bv \in \bCP^{\rm RT,0}_p(K),
     \label{Old_E_3^p_1}
     \\[5pt]
     \<\bu - (\bu_1 + \bu_2^p + \tilde\bu_3^p),\bcurl\,\phi\>_{0,K} = 0
     & \quad &
     \forall \phi \in \CP^{0}_p(K).
     \label{Old_E_3^p_2}
\eea

These interpolation operators satisfy the following properties.

\begin{prop} \label{pr_old_properties}
{\rm (cf. \cite[Propositions~1-3]{DemkowiczB_03_pIE}).}
\begin{itemize}

\item[$1^{\circ}$.]
For $r > 0$ the operators $\Pi_p^{1}:\; \bH^{1+r}(K) \rightarrow H^{1}(K)$
and $\Pi_p^{\div}:\; \bH^r(K) \cap \bH(\div,K) \rightarrow \bH(\div,K)$
are well defined and bounded, with corresponding operator norms independent
of the polynomial degree $p$.

\item[$2^{\circ}$.]
The operators $\Pi_p^{1}$ and $\Pi_p^{\div}$ preserve scalar polynomials in
$\CP_p(K)$ and polynomial vector fields in $\bCP^{\rm RT}_p(K)$, respectively.

\item[$3^{\circ}$.]
For $r > 0$, the diagram in (\ref{Old_deRham}) commutes.

\end{itemize}
\end{prop}

The next proposition gives optimal interpolation error estimates
for the operators $\Pi_p^{1}$ and $\Pi_p^{\div}$. These estimates
are proved in~\cite{BespalovH_OEE} (see Theorems~4.1 and~4.2 therein).

\begin{prop} \label{pr_old_estimates}

\begin{itemize}

\item[{\rm (i)}]
Let $g \in H^{1+r}(K)$, $r >0$. Then there exists a positive constant $C$
independent of $p$ and $g$ such that
\[  
    |g - \Pi_p^1\, g|_{H^1(K)} \le C\,p^{-r}\,\|g\|_{H^{1+r}(K)}.
\]  

\item[{\rm (ii)}]
Let $\bu \in \bH^r(\div, K)$, $r > 0$. Then there exists a positive constant $C$
independent of $p$ and $\bu$ such that
\[
  \|\bu - \Pi_p^{\div}\,\bu\|_{\bH(\div,K)} \le
  C\, p^{-r}\, \|\bu\|_{\bH^r(\div,K)}.
\]

\end{itemize}
\end{prop}

\section{Proofs of theorems} \label{sec_proofs}

In this section we prove the main results of the paper.

\bigskip

\noindent{\bf Proof of Theorem~\ref{thm_stab}.}
Let $\bu \in \bH^r(K) \cap \tilde\bH^{-1/2}(\div,K)$, $r > 0$.
We will study each term on the right-hand side of (\ref{E^p}).
Throughout the proof we denote by $s$ a small parameter such that
$0 < s < \min\,\{\frac 12,r\}$ for given $r>0$.

{\bf Step 1.} Fixing an edge $\ell \subset \partial K$ and using a function
\[
  \phi_{\ell} \in H^{1-s}(K),\quad
  \phi_{\ell} = \cases{
                         1 & \hbox{on $\ell$},\cr
                         0 & \hbox{on $\partial K \backslash \ell$}\cr
                        }
\]
as a test function, we integrate by parts to obtain
\beas
     \int\limits_{\ell} \bu \cdot \bn\,d\sigma
     & = &
     \int\limits_{\partial K} (\bu \cdot \bn)\,\phi_{\ell}\,d\sigma =
     \int\limits_{K} (\div\,\bu)\,\phi_{\ell}\,d\bx +
     \int\limits_{K} \bu \cdot \grad\phi_{\ell}\,d\bx
     \\[3pt]
     & \le &
     \|\div\,\bu\|_{\tH^{-1+s}(K)}\,\|\phi_{\ell}\|_{H^{1-s}(K)} +
     \|\bu\|_{\bH^s(K)}\,\|\grad\phi_{\ell}\|_{\bH^{-s}(K)}
     \\[3pt]
     & \le &
     C(\phi_{\ell},s)
     \left(
           \|\bu\|_{\bH^r(K)} + \|\div\,\bu\|_{\tH^{-1/2}(K)}
     \right).
\eeas
Note that if $\div\,\bu \in H^{-1+s}(K)$ then an extension to
$\div\,\bu \in \tH^{-1+s}(K)$ exists but is not unique. However,
by assumption $\div\,\bu \in \tH^{-1/2}(K) \subset \tH^{-1+s}(K)$,
which is a unique extension (see~\cite{Mikhailov_08_ATE} for details).
Thus, $\bu_1$ in (\ref{E_1}) is well defined.
Moreover, since $\bu_1$ is a lowest order interpolant, we find
by the equivalence of norms in finite-dimensional spaces that
\[  
    \|\bu_1\|_{\bH(\div,K)} \le
    C \sum\limits_{\ell \subset \partial K}
    \Big| \int\limits_{\ell} \bu \cdot \bn\,d\sigma\Big| \le
    C\left(\|\bu\|_{\bH^r(K)} + \|\div\,\bu\|_{\tH^{-1/2}(K)}\right).
\]  
Hence, due to the finite dimensionality of $\bu_1$, we obtain by using Lemma~\ref{lm_trace}
\bea
     \|(\bu - \bu_1) \cdot \bn\|_{H^{-1/2+s}(\partial K)}
     & \le &
     C \left(
             \|\bu - \bu_1\|_{\bH^s(K)} + \|\div (\bu - \bu_1)\|_{\tH^{-1+s}(K)}
      \right)
     \nonumber
     \\[3pt]
     & \le &
     C \left(
             \|\bu\|_{\bH^s(K)} + \|\div\,\bu\|_{\tH^{-1+s}(K)} +
             \|\bu_1\|_{\bH(\div,K)}
      \right)
     \nonumber
     \\[3pt]
     & \le &
     C \left(
             \|\bu\|_{\bH^r(K)} + \|\div\,\bu\|_{\tH^{-1/2}(K)}
      \right).
     \label{stab2}
\eea

{\bf Step 2.} From the construction of $\bu_1$ and from the result of
Step~1 we conclude that
\[
  (\bu - \bu_1) \cdot \bn \in H^{-1/2+s}(\partial K),\quad
  \int\limits_{\partial K} (\bu - \bu_1) \cdot \bn\,d\sigma = 0.
\]
Therefore, due to the isomorphism (see \cite[Lemma~2]{DemkowiczB_03_pIE})
\[
  \frac{\partial}{\partial \sigma}: H^{1/2+s}(\partial K)/{\field{R}} \rightarrow
  H^{-1/2+s}_{*}(\partial K) =
  \{\phi \in H^{-1/2+s}(\partial K);\; \<u,1\>_{0,\partial K} = 0\},
\]
the function $\psi$ in (\ref{psi}) is well defined, $\psi \in H^{1/2+s}(\partial K)$,
$\psi|_{\ell} \in \tH^{1/2}(\ell)$ for any edge $\ell \subset \partial K$,
and
\be \label{stab3}
    \sum\limits_{\ell \subset \partial K}
    \|\psi|_{\ell}\|_{\tH^{1/2}(\ell)} \le
    C\,\sum\limits_{\ell \subset \partial K}
    \|\psi|_{\ell}\|_{H^{1/2+s}_0(\ell)} \le
    C\,\|\psi\|_{H^{1/2+s}(\partial K)} \le
    C\,\|(\bu - \bu_1) \cdot \bn\|_{H^{-1/2+s}(\partial K)}.
\ee
Hence, (\ref{psi_2^l}) is uniquely solvable and
\be \label{stab4}
    \|\psi_2^{\ell}\|_{\tH^{1/2}(\ell)} \le
    C\, \|\psi|_{\ell}\|_{\tH^{1/2}(\ell)}.
\ee
Furthermore, applying the polynomial extension result of
Babu{\v s}ka {\em et al.} \cite{BabuskaCMP_91_EPp}, we find the desired
polynomial $\psi_{2,p}^{\ell} \in \CP_p(K)$ (see (\ref{psi_2,p^l}))
satisfying
\be \label{stab5}
    \|\psi_{2,p}^{\ell}\|_{H^{1}(K)} \le
    C\, \|\psi_2^{\ell}\|_{\tH^{1/2}(\ell)}.
\ee
Thus, $\bu_2^p$ in (\ref{E_2^p}) is well defined.
Putting together (\ref{stab3})--(\ref{stab5}) we find
\beas
     \|\bu_2^p\|_{0,K} \le
     C \sum\limits_{\ell \subset \partial K}
     \|\bcurl\,\psi_{2,p}^{\ell}\|_{0,K} \le
     C \sum\limits_{\ell \subset \partial K}
     \|\psi_{2,p}^{\ell}\|_{H^1(K)} \le
     C\,\|(\bu - \bu_1) \cdot \bn\|_{H^{-1/2+s}(\partial K)}.
\eeas
Hence, making use of (\ref{stab2}), we obtain
\be \label{stab6}
    \|\bu_2^p\|_{\bH(\div,K)} = \|\bu_2^p\|_{0,K} \le
    C \left(
            \|\bu\|_{\bH^r(K)} + \|\div\,\bu\|_{\tH^{-1/2}(K)}
     \right).
\ee
{\bf Step 3.} The vector bubble function $\bu_3^p$ is uniquely defined by
(\ref{E_3^p_1})--(\ref{E_3^p_2}). To estimate the norms of $\bu^p_3$ and
$\div\,\bu^p_3$ we use the discrete Helmholtz decomposition
\be \label{stab6_1}
    \bu_3^p = \bv_p + \bcurl\,\phi_p,
\ee
where $\phi_p \in \CP_p^0(K)$ and $\bv_p \in \bCP^{\rm RT,0}_p(K)$ is such that
$\<\bv_p,\bcurl\,\varphi\>_{0,K} = 0$ for all $\varphi \in \CP^0_p(K)$.

From (\ref{E_3^p_1}) one has by using the result of Step~1
\bea \label{stab7}
     \|\div\,\bu_3^p\|_{\tH^{-1/2}(K)}
     & \le &
     C\,\|\div(\bu - \bu_1)\|_{\tH^{-1/2}(K)} \le
     C\left(\|\div\,\bu\|_{\tH^{-1/2}(K)} + |\div\,\bu_1|\right)
     \nonumber
     \\[3pt]
     & \le &
     C\left(\|\bu\|_{\bH^{r}(K)} + \|\div\,\bu\|_{\tH^{-1/2}(K)}\right).
\eea
Then, applying Lemma~\ref{lm_Friedrichs}(i) and
recalling that $\div\,\bv_p = \div\,\bu^p_3$, we find
\be \label{stab8}
    \|\bv_p\|_{0,K} \le
    C\,\|\div\,\bv_p\|_{\tH^{-1}(K)} \le C\,\|\div\,\bu^p_3\|_{\tH^{-1/2}(K)}.
\ee
Since $\<\bv_p,\bcurl\,\phi_p\>_{0,K} = 0$, we estimate the norm of
$\bcurl\,\phi_p$ by using (\ref{E_3^p_2}) and by employing the results
of the first two steps:
\be \label{stab9}
    \|\bcurl\,\phi_p\|_{0,K} \le
    \|\bu - \bu_1 - \bu_2^p\|_{0,K} \le
    C\left(\|\bu\|_{\bH^{r}(K)} + \|\div\,\bu\|_{\tH^{-1/2}(K)}\right).
\ee
Combining (\ref{stab7})--(\ref{stab9}) and applying the triangle inequality
we obtain, by making use of decomposition (\ref{stab6_1}),
\[
  \|\bu_3^p\|_{0,K} + \|\div\,\bu_3^p\|_{\tH^{-1/2}(K)} \le
  C \left(\|\bu\|_{\bH^{r}(K)} + \|\div\,\bu\|_{\tH^{-1/2}(K)}\right).
\]
The boundedness of the operator $\Pi_p^{\div,-\frac 12}$ (and
inequality (\ref{stab})) now follows by putting together the results
of the three individual steps and by applying the triangle inequality.

The polynomial-preserving property of the operator $\Pi_p^{\div,-\frac 12}$
easily follows from its definition.~\qed

It is essential for the proof of Theorem~\ref{thm_deRham} given below
that the $\tH^{-1/2}(K)$-inner product satisfies (\ref{key_ip_property})
(i.e., reduces to the $L^2(K)$-inner product for a constant function).
As it follows from Lemma~\ref{lm_tilde_-1/2_ip} in the Appendix,
the $\tH^{-1/2}(K)$-inner product given by (\ref{tilde_-1/2-ip})
satisfies this property.

\bigskip

\noindent{\bf Proof of Theorem~\ref{thm_deRham}.}
To prove the first part of the diagram, we consider $\bu = \bcurl\,g$, $g \in H^{1+r}(K)$.
Let us decompose $\Pi^{1}_p g$ and $\Pi^{\div,-\frac 12}_p \bu$ as in (\ref{int_4}) and
(\ref{E^p}), respectively.
Then it follows from the definitions of these interpolation operators that
$\bu_1 = \bcurl\,g_1$ and $\bu_2^p = \bcurl\,g_2^p$ (cf. \cite{DemkowiczB_03_pIE}).
Hence, $\div\,\bu = \div\,\bu_1 = \div\,\bu_2^p = 0$ and it follows from
(\ref{E_3^p_1}) that $\div\,\bu_3^p = 0$. Therefore, decomposing $\bu_3^p$ as in
(\ref{stab6_1}) and comparing (\ref{E_3^p_2}) with (\ref{int_3}), we conclude
that $\bu_3^p = \bcurl\,g_3^p$.
Thus, $\Pi^{\div,-\frac 12}_p (\bcurl\,g) = \bcurl(\Pi^{1}_p g)$.

Let us prove the second part of the diagram.
For any $\varphi \in \CP_{p-1}(K)$ there exists $\bv_p \in \bCP^{\rm RT}_p(K)$
such that $\div\,\bv_p = \varphi$. Therefore, decomposing $\Pi^{\div,-\frac 12}_p \bu$
as in (\ref{E^p}), we need to show that for all $\bv_p \in \bCP^{\rm RT}_p(K)$
there holds
\be \label{commute1}
    \Big\<\div\Big(\bu - \Pi^{\div,-\frac 12}_p \bu\Big), \div\,\bv_p\Big\>_{\tH^{-1/2}(K)} =
    \<\div(\bu - (\bu_1 + \bu_3^p)), \div\,\bv_p\>_{\tH^{-1/2}(K)} = 0.
\ee
Let us also decompose $\bv_p = \Pi^{\div,-\frac 12}_p \bv_p \in \bCP_p^{\rm RT}(K)$
as in (\ref{E^p}):
\[
  \bv_p = \bv_1 + \bv_2^p + \bv_3^p,\quad
  \div\,\bv_1 = \hbox{const},\quad \div\,\bv_2^p = 0,\quad
  \bv_3^p \in \bCP^{\rm RT,0}_p(K).
\]
Then, recalling (\ref{E_3^p_1}), applying Lemma~\ref{lm_tilde_-1/2_ip},
and integrating by parts, we prove (\ref{commute1}):
\beas
     \lefteqn{
     \Big\<\div\Big(\bu - \Pi^{\div,-\frac 12}_p \bu\Big), \div\,\bv_p\Big\>_{\tH^{-1/2}(K)}
     }
     \\
     &\qquad\qquad = &
     \<\div(\bu - \bu_1 - \bu_3^p), \hbox{const}\>_{\tH^{-1/2}(K)} +
     \Big\<\div\Big(\bu - \Pi^{\div,-\frac 12}_p \bu\Big), \div\,\bv_3^p\Big\>_{\tH^{-1/2}(K)}
     \\
     &\qquad\qquad = &
     \<\div(\bu - \bu_1 - \bu_3^p), \hbox{const}\>_{0,K} =
     \hbox{const}\, \int\limits_{\partial K} (\bu - \bu_1 - \bu_3^p) \cdot \bn\,d\sigma = 0.
\eeas
For the last step we used the fact that $\bu_3^p \cdot \bn|_{\partial K} = 0$
and then applied (\ref{E-E_1}).\qed

For the proof of Theorem~\ref{thm_err_estimate} we will need two
auxiliary results regarding the new $\bH(\div)$-conforming
interpolation operator $\Pi^{\div,-\frac 12}_p$.
These results are formulated in the next two lemmas:
the first one concerns the normal trace of the interpolant
$\Pi^{\div,-\frac 12}_p\,\bu$ on the boundary $\partial K$,
and the second one states some auxiliary error estimates for $\Pi^{\div,-\frac 12}_p$.

\begin{lemma} \label{lm_aux1}
Let $\bu \in \bH^r(K) \cap \tilde\bH^{-1/2}(\div,K)$ with $r > 0$, and let 
$\bu^p = \Pi^{\div,-\frac 12}_p \bu \in \bCP^{\rm RT}_p(K)$. Then for any edge
$\ell \subset \partial K$ there holds
\be \label{aux1}
    \|(\bu - \bu^p) \cdot \bn\|_{\tH^{-1}(\ell)} \le
    C\,p^{-1/2}\,\|(\bu - \bu^p) \cdot \bn\|_{H^{-1/2}(\partial K)}.
\ee
\end{lemma}

\begin{proof}
If $\bu \in \bH^r(K) \cap \bH(\div,K)$ with $r > 0$,
then it was proved in \cite[Lemma~3.3]{BespalovH_hpA} that
\[
  \|(\bu - \Pi_p^{\div} \bu) \cdot \bn\|_{\tH^{-1}(\ell)} \le
  C\,p^{-1/2}\,\|(\bu - \Pi_p^{\div} \bu) \cdot \bn\|_{H^{-1/2}(\partial K)}.
\]
We note, however, that 
$\Pi^{\div}_p \bu \cdot \bn = \Pi^{\div,-1/2}_p \bu \cdot \bn =
 (\bu_1 +\bu_2^p) \cdot \bn$ on the boundary $\partial K$, and,
as it follows from the proof of Theorem~\ref{thm_stab} above,
$\bu_1$ and $\bu_2^p$ are in fact well defined for
$\bu \in \bH^r(K) \cap \tilde\bH^{-1/2}(\div,K)$, $r > 0$.
Therefore, the proof of Lemma~3.3 in~\cite{BespalovH_hpA}
carries over to the case considered in this paper and inequality
(\ref{aux1}) is valid.
\end{proof}

\begin{lemma} \label{lm_aux_err}
Let $r > 0$ and $s > \max\{-\frac 12,r-1\}$.
If $\bu \in \bH^r(K)$ and $\div\,\bu \in H^s(K)$, then
\be \label{aux_err_1}
    \|\div(\bu - \Pi^{\div,-\frac 12}_p\,\bu)\|_{\tH^{-1/2}(K)} \le
    C\,p^{-(1/2+s)}\,\|\div\,\bu\|_{H^s(K)}
\ee
and for any $\eps > 0$ there holds
\be \label{aux_err_2}
    \|\bu - \Pi^{\div,-\frac 12}_p\,\bu\|_{0,K} \le
    C \left(
            p^{-r}\,\|\bu\|_{\bH^r(K)} +
            \eps^{-1}\,p^{-(s+1/2-\eps)}\|\div\,\bu\|_{H^s(K)}
     \right).
\ee
The positive constants $C$ in {\rm (\ref{aux_err_1})} and {\rm (\ref{aux_err_2})}
are independent of $\bu$ and $p$.
\end{lemma}

\begin{proof}
Estimate (\ref{aux_err_1}) is an immediate consequence of
the commuting diagram property (\ref{deRham}) and
Lemma~\ref{lm_tH^{-1/2}_estimate}:
\[
  \|\div\,\bu - \div(\Pi^{\div,-\frac 12}_p\,\bu)\|_{\tH^{-1/2}(K)} =
  \|\div\,\bu - \Pi^{-1/2}_{p-1}(\div\,\bu)\|_{\tH^{-1/2}(K)} \le
  C\,p^{-(1/2+s)}\,\|\div\,\bu\|_{H^s(K)}.
\]
Let us now prove (\ref{aux_err_2}). For $p = 1$ this estimate follows trivially from
Theorem~\ref{thm_stab}. Let $p \ge 2$.
Using Lemma~\ref{lm_Poincare} we decompose $\bu$ as follows:
\be \label{aux_err_3}
    \bu = \bcurl\,\psi + \bv,\qquad
    \psi \in H^{r+1}(K),\ \ \bv \in \bH^{s+1}(K).
\ee
Moreover, the norms of $\bv$ and $\psi$ are bounded as in (\ref{P_2}).
Then, applying the interpolation operator $\Pi_p^{\div,-\frac 12}$
and using its commutativity with $\Pi^1_{p}$ (see (\ref{deRham})), we write
\be \label{aux_err_4}
    \Pi_p^{\div,-\frac 12}\,\bu =
    \Pi_p^{\div,-\frac 12}(\bcurl\,\psi) + \Pi_p^{\div,-\frac 12}\,\bv =
    \bcurl(\Pi_{p}^1 \psi) + \Pi_p^{\div,-\frac 12}\,\bv.
\ee
Since $\Pi_p^{\div,-\frac 12}$ is a bounded operator preserving polynomials
(see Theorem~\ref{thm_stab}), one has for any polynomial
$\bv_p \in (\CP_{p-1}(K))^2 \subset \bCP_p^{\rm RT}(K)$:
\bea
     \|\bv - \Pi_p^{\div,-\frac 12}\,\bv\|_{0,K}
     & = &
     \|\bv - \bv_p - \Pi_p^{\div,-\frac 12}(\bv - \bv_p)\|_{0,K}
     \nonumber
     \\[5pt]
     & \le &
     C\, \inf_{\bv_p \in (\CP_{p-1}(K))^2}
     \Big(
          \|\bv - \bv_p\|_{\bH^{\tilde\eps}(K)} + \|\div(\bv - \bv_p)\|_{\tH^{-1/2}(K)}
     \Big)
     \nonumber
     \\[5pt]
     & \le &
     C\,\eps^{-1}\,\inf_{\bv_p \in (\CP_{p-1}(K))^2} \|\bv - \bv_p\|_{\bH^{1/2+\eps}(K)},
     \label{aux_err_5}
\eea
where $\tilde\eps \in (0,\frac 12)$ is fixed, $\eps > 0$ is arbitrarily small, and
for the last step we used Lemma~5 of \cite{Heuer_01_ApS} as well as the boundedness
of the divergence operator to estimate
\beas
     \|\div(\bv - \bv_p)\|_{\tH^{-1/2}(K)}
     & \le &
     \|\div(\bv - \bv_p)\|_{\tH^{-1/2+\eps}(K)}
     \\[3pt]
     & \le &
     C\,\eps^{-1}\,\|\div(\bv - \bv_p)\|_{H^{-1/2+\eps}(K)} \le
     C\,\eps^{-1}\,\|\bv - \bv_p\|_{\bH^{1/2+\eps}(K)}.
\eeas
Applying now Lemma~\ref{lm_2D_p-approx} componentwise and using the first
inequality in (\ref{P_2}), we deduce from (\ref{aux_err_5}) that
\be \label{aux_err_6}
    \|\bv - \Pi_p^{\div,-\frac 12}\,\bv\|_{0,K} \le
    C\,\eps^{-1}\,(p-1)^{-(s+1/2-\eps)}\,\|\bv\|_{\bH^{s+1}(K)} \le
    C\,\eps^{-1}\,p^{-(s+1/2-\eps)}\,\|\div\,\bu\|_{H^s(K)}.
\ee
On the other hand, applying Proposition~\ref{pr_old_estimates}(i) and
the second inequality in (\ref{P_2}) we obtain
\be \label{aux_err_7}
    \|\bcurl(\psi - \Pi_{p}^{1}\psi)\|_{0,K} =
    |\psi - \Pi_{p}^{1}\psi|_{H^1(K)} \le
    C\,p^{-r}\,\|\psi\|_{H^{1+r}(K)} \le
    C\,p^{-r}\,\|\bu\|_{\bH^{r}(K)}.
\ee
Combining (\ref{aux_err_6}) and (\ref{aux_err_7}) we prove
(\ref{aux_err_2}) by making use of decompositions (\ref{aux_err_3}),
(\ref{aux_err_4}) and the triangle inequality.
\end{proof}

Now we are in a position to prove the main interpolation error estimate.

\bigskip

\noindent{\bf Proof of Theorem~\ref{thm_err_estimate}.}
For simplicity of notation we denote
$\bu^p := \Pi^{\div,-\frac 12}_p\,\bu \in \CP^{\rm RT}_p(K)$.
Let us consider an auxiliary problem: find $\bu_0 \in \bH(\div,K)$ such that
\bea
    & \<\bu - \bu_0,\bv\>_{0,K} + \<\div(\bu - \bu_0),\div\,\bv\>_{0,K} = 0
    \quad
    \forall\,\bv \in \bH_0(\div,K), &
    \label{estim_proof_1}
    \\[5pt]
    & \bu_0 \cdot \bn = \bu^p \cdot \bn\qquad
    \hbox{on $\partial K$}. &
    \nonumber
\eea
Then, using Lemma~4.8 in \cite{BuffaC_03_EFI} and applying
Lemmas~\ref{lm_aux1} and~\ref{lm_trace}, we estimate for
$t={-}1,\,-\frac 12$
\bea
    \|\bu - \bu_0\|_{\tilde\bH^{t+1/2}(\div,K)}
    & \le &
    C\,\|(\bu - \bu^p) \cdot \bn\|_{H^{t}(\partial K)} \le
    C\,p^{t+1/2}\,\|(\bu - \bu^p) \cdot \bn\|_{H^{-1/2}(\partial K)}\ \qquad
    \nonumber
    \\
    & \le &
    C\,p^{t+1/2}\left(
                      \|\bu {-} \bu^p\|_{0,K} + \|\div(\bu {-} \bu^p)\|_{\tH^{-1/2}(K)}
               \right).
    \label{estim_proof_2}
\eea
By the triangle inequality one has
\be \label{estim_proof_3}
    \|\bu - \bu^p\|_{\tilde\bH^{-1/2}(\div,K)} \le
    \|\bu - \bu_0\|_{\tilde\bH^{-1/2}(\div,K)} +
    \|\bu_0 - \bu^p\|_{\tilde\bH^{-1/2}(\div,K)}.
\ee
For the first term on the right-hand side of (\ref{estim_proof_3})
we have by using (\ref{estim_proof_2}) with $t = -1$:
\be \label{estim_proof_4}
    \|\bu - \bu_0\|_{\tilde\bH^{-1/2}(\div,K)} \le
    C\,p^{-1/2}\left(
                     \|\bu - \bu^p\|_{0,K} + \|\div(\bu - \bu^p)\|_{\tH^{-1/2}(K)}
              \right).
\ee
Now, we consider the second term on the right-hand side of (\ref{estim_proof_3})
and prove that
\be \label{estim_proof_5}
    \|\bu_0 - \bu^p\|_{\tilde\bH^{-1/2}(\div,K)} \le
    C \left(
            p^{-1/2}\,\|\bu - \bu^p\|_{0,K} + \|\div(\bu - \bu^p)\|_{\tH^{-1/2}(K)}
     \right).
\ee
Denote $\bX := \tilde\bH^{-1/2}_0(\div,K)$, and let
$\bX'$ be the dual space of $\bX$ (with $\bL^2(K)$ as pivot space).
From \cite[Section~6]{BuffaC_01_TII} we know that any $\bw \in \bX'$ can be
decomposed as follows:
\[
  \bw = \grad\,f + \bcurl\,g,\qquad
  f \in H^{1/2}(K)/{\field{R}},\
  g \in H^1_0(K) \cap H^{3/2}(K),
\]
and
\be \label{estim_proof_6}
    \|f\|_{H^{1/2}(K)/{\field{R}}} + \|g\|_{H^{3/2}(K)} \le
    C\,\|\bw\|_{\bX'}.
\ee
Hence, recalling that $\bu_0 - \bu^p \in \bH_0(\div,K) \subset \bX$, we have
\bea
    \|\bu_0 - \bu^p\|_{\tilde\bH^{-1/2}(\div,K)}
    & = &
    \sup_{\bzero \not= \bw \in \bX'}
    {\<\bu_0 - \bu^p, \bw\>_{0,K}
    \over{\|\bw\|_{\bX'}}} =
    \sup_{\bzero \not= \bw \in \bX'}
    {\<\bu_0 - \bu^p, \grad\,f + \bcurl\,g\>_{0,K}
    \over{\|\bw\|_{\bX'}}}
    \nonumber
    \\[3pt]
    & = &
    \sup_{\bzero \not= \bw \in \bX'}
    {-\<\div(\bu_0 - \bu^p), f\>_{0,K} +
     \<\bu_0 - \bu^p, \bcurl\,g\>_{0,K}
    \over{\|\bw\|_{\bX'}}}.
    \label{estim_proof_7}
\eea
Let $g_p \in \CP_p^0(K)$. Using (\ref{E_3^p_2}) with $\phi = g_p$ and
(\ref{estim_proof_1}) with $\bv = \bcurl\,g_p$, we find that
\[
  \<\bu_0 - \bu^p,\bcurl\,g_p\>_{0,K} =
  \<\bu - \bu^p,\bcurl\,g_p\>_{0,K} -
  \<\bu - \bu_0,\bcurl\,g_p\>_{0,K}   = 0\quad
  \forall\,g_p \in \CP_p^0(K).
\]
Therefore, selecting $g_p := \Pi_p^1 g \in \CP_p^0(K)$ and using
Proposition~\ref{pr_old_estimates}(i), we obtain from (\ref{estim_proof_7})
\bea
    \lefteqn{
             \|\bu_0 - \bu^p\|_{\tilde\bH^{-1/2}(\div,K)}
            }
    \nonumber
    \\[3pt]
    & \le &
    \sup_{\bzero \not= \bw \in \bX'}
    {\|\div(\bu_0 - \bu^p)\|_{\tH^{-1/2}(K)}\, \|f\|_{H^{1/2}(K)/{\field{R}}} +
     \|\bu_0 - \bu^p\|_{0,K}\, |g - \Pi_p^1 g|_{H^1(K)}
    \over{\|\bw\|_{\bX'}}}
    \nonumber
    \\[3pt]
    & \le &
    \sup_{\bzero \not= \bw \in \bX'}
    {\|\div(\bu_0 - \bu^p)\|_{\tH^{-1/2}(K)}\, \|f\|_{H^{1/2}(K)/{\field{R}}} +
     C\,p^{-1/2}\,\|\bu_0 - \bu^p\|_{0,K}\, \|g\|_{H^{3/2}(K)}
    \over{\|\bw\|_{\bX'}}}
    \nonumber
    \\[3pt]
    & \stackrel{(\ref{estim_proof_6})}{\le} &
    C \left(
            \|\div(\bu_0 - \bu^p)\|_{\tH^{-1/2}(K)} + p^{-1/2}\,\|\bu_0 - \bu^p\|_{0,K}
     \right).
    \label{estim_proof_8}
\eea
Both norms on the right-hand side of (\ref{estim_proof_8}) are estimated by applying
the triangle inequality and inequalities (\ref{estim_proof_2})
(with $t = -1$ and $t = -\frac 12$, respectively):
\bea
    \|\div(\bu_0 - \bu^p)\|_{\tH^{-1/2}(K)}
    & \le &
    \|\bu - \bu_0\|_{\tilde\bH^{-1/2}(\div,K)} +
    \|\div(\bu - \bu^p)\|_{\tH^{-1/2}(K)}
    \nonumber
    \\
    & \le &
    C \left(
            p^{-1/2}\,\|\bu - \bu^p\|_{0,K} + \|\div(\bu - \bu^p)\|_{\tH^{-1/2}(K)}
     \right)
    \label{estim_proof_9}
\eea
and
\bea
    \|\bu_0 - \bu^p\|_{0,K}
    & \le &
    \|\bu - \bu^p\|_{0,K} +
    \|\bu - \bu_0\|_{\bH(\div,K)}
    \nonumber
    \\
    & \le &
    C \left(
            \|\bu - \bu^p\|_{0,K} + \|\div(\bu - \bu^p)\|_{\tH^{-1/2}(K)}
     \right).
    \label{estim_proof_10}
\eea
The desired inequality in (\ref{estim_proof_5}) then follows from
(\ref{estim_proof_8})-(\ref{estim_proof_10}).

Now, collecting (\ref{estim_proof_4}) and (\ref{estim_proof_5}) in (\ref{estim_proof_3}),
we obtain
\[
  \|\bu - \bu^p\|_{\tilde\bH^{-1/2}(\div,K)} \le
  C \left(
          p^{-1/2}\,\|\bu - \bu^p\|_{0,K} + \|\div(\bu - \bu^p)\|_{\tH^{-1/2}(K)}
   \right).
\]
Hence, recalling that $\bu^p = \Pi^{\div,-\frac 12}_p\,\bu$ and
applying Lemma~\ref{lm_aux_err} with $s = r$ and $\eps = \frac 12$,
we arrive at estimate (\ref{err_estimate}).\qed

\appendix
\section{Some equivalent norms and corresponding inner products
         in the Sobolev spaces $H^r$ and $\tH^r$ for $r = \pm\frac 12$} \label{app_norms}
\setcounter{equation}{0}

In this appendix we consider the Sobolev spaces $H^r$ and $\tH^r$
on the reference element $K$ for $r = \pm\frac 12$.
We will derive expressions for norms which are
equivalent to those defined in Section~\ref{sec_spaces}.
First, let us introduce some notation.

\begin{enumerate}

\item[$1^{\circ}$.]
We denote by $D$ the polyhedron (cube or triangular prism) such that
$D = K \times (0,1)$. Thus $\partial D = \cup_{i=1}^{\CI} \bar\G_i$
($\CI=5$ if $K=T$ and $\CI=6$ if $K=Q$).
Let $K = \G_1 = \{(x_1,x_2,0);\; (x_1,x_2) \in K\}$,
$\G_{\CI} = \{(x_1,x_2,1);\; (x_1,x_2) \in K\}$, and denote
$\tilde K = \partial D \backslash \bar\G_{\CI}$.
Note that $\tilde K$ is an open surface.
We will denote by $\bnu$ the outward normal unit vector to $\partial D$, and
we will use the standard notation for the gradient $\nabla$ and for the Laplace
operator $\Delta$, both acting on scalar functions of three variables.

\item[$2^{\circ}$.]
Given $u \in H^{-1/2}(K)$, we denote by $\tilde u_K$ the solution of the mixed
problem: find $\tilde u_K \in H^1(D)$ such that
\[
  \Delta \tilde u_K = 0\ \hbox{in $D$},\quad
  \hbox{$\frac{\partial\tilde u_K}{\partial \bnu} = u$ on $K$},\quad
  \tilde u_K = 0\ \hbox{on $\partial D \backslash K$}.
\]
If $u \in H^{-1/2}(\tilde K)$, then we will use the same notation as above
with $K$ replaced by $\tilde K$.

\item[$3^{\circ}$.]
Given $u \in H^{1/2}(\partial D)$, we denote by $\tilde{\tilde u}$
its harmonic extension, i.e., the solution of the Dirichlet problem:
find $\tilde{\tilde u} \in H^1(D)$ such that
\be \label{Dir_ext}
    \Delta \tilde{\tilde u} = 0\ \hbox{in $D$},\quad
    \tilde{\tilde u} = u\ \hbox{on $\partial D$}.
\ee

\item[$4^{\circ}$.]
Given $u \in \tH^{1/2}(K)$, we denote by $u^{\circ}$ the extension
of $u$ by zero onto $\partial D$.
Thus, $u^{\circ} \in H^{1/2}(\partial D)$.

\end{enumerate}

We make use of standard definitions for the norm and the semi-norm in $H^1(D)$:
\[
   \|u\|_{H^1(D)} =
   \left(\|u\|_{0,D}^2 + |u|_{H^1(D)}^2\right)^{1/2},\quad
   |u|_{H^1(D)} = \|\grad u\|_{0,D}.
\]
Since $H^{1/2}(\partial D)$ is the trace space of $H^1(D)$,
the norm and the semi-norm in $H^{1/2}(\partial D)$ can be equivalently
written as follows
\bea
    \|u\|_{H^{1/2}(\partial D)}
    & \simeq &
    \stack{U|_{\partial D}=u}{\inf_{U \in H^1(D)}} \|U\|_{H^1(D)},
    \nonumber
    \\
    |u|_{H^{1/2}(\partial D)}
    & \simeq &
    \stack{U|_{\partial D}=u}{\inf_{U \in H^1(D)}} |U|_{H^1(D)} =
    \|\grad \tilde{\tilde u}\|_{0,D}.
    \label{1/2-seminorm}
\eea
Now we can define equivalent norms in $\tilde H^{1/2}(K)$ and $H^{1/2}(K)$:
\bea
    \|u\|_{\tH^{1/2}(K)}
    & \simeq &
    |u^\circ|_{H^{1/2}(\partial D)} \simeq
    \Big\|\grad \widetilde{\widetilde{u^\circ}}\Big\|_{0,D},
    \label{tilde_1/2-norm}
    \\
    \|u\|_{H^{1/2}(K)}
    & \simeq &
    \stack{U|_{K}=u}{\inf_{U \in \tilde H^{1/2}(\tilde K)}} \|U\|_{\tH^{1/2}(\tilde K)},
    \label{1/2-norm}
\eea
where $\|\cdot\|_{\tH^{1/2}(\tilde K)}$ is defined as in (\ref{tilde_1/2-norm}),
because $\tilde K$ is an open surface.

From (\ref{tilde_1/2-norm}) one can easily derive the expression for the
corresponding $\tH^{1/2}(K)$-inner product.
In fact, applying the parallelogram law twice, integrating by parts, and
recalling notations $3^\circ$, $4^\circ$, we find
(see also \cite{DemkowiczB_03_pIE})
\bea
    \<u,v\>_{\tH^{1/2}(K)}
    & = &
    \Big\<\grad \ttu,\grad \ttv\Big\>_{0,D} =
    \Big\<\frac{\partial\ttu}{\partial \bnu},\ttv\Big\>_{0,\partial D} =
    \nonumber
    \\[5pt]
    & = &
    \Big\<\frac{\partial\ttu}{\partial \bnu},v\Big\>_{0,K} =
    \Big\<u,\frac{\partial\ttv}{\partial \bnu}\Big\>_{0,K}\quad
    \forall u,v \in \tH^{1/2}(K).
    \label{tilde_1/2-ip}
\eea

The space $H^{-1/2}(K)$ is the dual space of $\tH^{1/2}(K)$.
We prove the following result regarding an equivalent norm in $H^{-1/2}(K)$.

\begin{lemma} \label{lm_-1/2-norm}
For any $u \in H^{-1/2}(K)$ there holds
\be \label{-1/2-norm}
    \|u\|_{H^{-1/2}(K)} \simeq \|\grad \tilde u_K\|_{0,D}.
\ee
The $H^{-1/2}$-inner product corresponding to the norm on the
right-hand side of {\rm (\ref{-1/2-norm})} reads as
\be \label{-1/2_ip}
    \<u,v\>_{H^{-1/2}(K)} = \<u,\tilde v_K\>_{0,K} = \<\tilde u_K,v\>_{0,K}\quad
    \forall u,v \in H^{-1/2}(K).
\ee
\end{lemma}

\begin{proof}
Using notations $2^\circ-4^\circ$,
we integrate by parts to obtain for any $u \in H^{-1/2}(K)$ and
any $v \in \tH^{1/2}(K)$
\[
      \Big\<\grad \tilde u_K,\grad \ttv\Big\>_{0,D} =
      \Big\<\frac{\partial\tilde u_K}{\partial \bnu},\ttv\Big\>_{0,\partial D} =
      \Big\<\frac{\partial\tilde u_K}{\partial \bnu},\ttv\Big\>_{0,K} +
      \Big\<\frac{\partial\tilde u_K}{\partial \bnu},\ttv\Big\>_{0,\partial D\backslash K} =
      \<u,v\>_{0,K}.
\]
Hence, we find from (\ref{dual_norm}) and (\ref{tilde_1/2-norm})
\be \label{-1/2-norm_aux1}
    \|u\|_{H^{-1/2}(K)} =
    \sup_{0\not=v\in\tH^{1/2}(K)}
    {\Big|\Big\<\grad \tilde u_K,\grad \ttv\Big\>_{0,D}\Big|
    \over{\|v\|_{\tH^{1/2}(K)}}} \simeq
    \sup_{0\not=v\in\tH^{1/2}(K)}
    {\Big|\Big\<\grad \tilde u_K,\grad \ttv\Big\>_{0,D}\Big|
    \over{\Big\|\grad \ttv\Big\|_{0,D}}}.
\ee
Let $w := \tilde u_K|_K$. One has $w \in \tH^{1/2}(K)$
because $\tilde u_K = 0$ on $\partial D\backslash K$. Moreover,
$w^\circ = \tilde u_K|_{\partial D}$ and, due to the uniqueness of the
solution to the Dirichlet problem (\ref{Dir_ext}), we conclude that
$\widetilde{\widetilde{w^\circ}} = \tilde u_K$. Therefore,
\be \label{-1/2-norm_aux2}
    \sup_{0\not=v\in\tH^{1/2}(K)}
    {\Big|\Big\<\grad \tilde u_K,\grad \ttv\Big\>_{0,D}\Big|
    \over{\Big\|\grad \ttv\Big\|_{0,D}}} \ge
    {\Big|\Big\<\grad \tilde u_K,\grad \widetilde{\widetilde{w^\circ}}\Big\>_{0,D}\Big|
    \over{\Big\|\grad \widetilde{\widetilde{w^\circ}}\Big\|_{0,D}}} =
    \|\grad \tilde u_K\|_{0,D}.
\ee
On the other hand, it is easy to see that
\be \label{-1/2-norm_aux3}
    \sup_{0\not=v\in\tH^{1/2}(K)}
    {\Big|\Big\<\grad \tilde u_K,\grad \ttv\Big\>_{0,D}\Big|
    \over{\Big\|\grad \ttv\Big\|_{0,D}}} \le
    \|\grad \tilde u_K\|_{0,D}.
\ee
Now (\ref{-1/2-norm}) immediately follows from 
(\ref{-1/2-norm_aux1})--(\ref{-1/2-norm_aux3}).

Using (\ref{-1/2-norm}) together with the parallelogram law we find
\[
  \<u,v\>_{H^{-1/2}(K)} =
  \Big\<\grad \tilde u_K,\grad \tilde v_K\Big\>_{0,D}\quad
  \forall u,v \in H^{-1/2}(K).
\]
Hence, integrating by parts and using notation~$2^\circ$, we derive (\ref{-1/2_ip}).
\end{proof}

The following lemma states an analogous result for the space $\tH^{-1/2}(K)$
which is the dual space of $H^{1/2}(K)$.

\begin{lemma} \label{lm_tilde-1/2-norm}
For any $u \in \tH^{-1/2}(K)$ there holds
\be \label{tilde-1/2-norm}
    \|u\|_{\tH^{-1/2}(K)} \simeq
    \Big\|\grad \widetilde{(u^\circ)}_{\tilde K}\Big\|_{0,D}.
\ee
The $\tH^{-1/2}$-inner product corresponding to the norm on the
right-hand side of {\rm (\ref{tilde-1/2-norm})} reads as
\be \label{tilde_-1/2-ip}
    \<u,v\>_{\tH^{-1/2}(K)} =
    \Big\<u,\widetilde{(v^\circ)}_{\tilde K}\Big\>_{0,K} =
    \Big\<\widetilde{(u^\circ)}_{\tilde K},v\Big\>_{0,K}\quad
    \forall u,v \in \tH^{-1/2}(K).
\ee
\end{lemma}

\begin{proof}
Let $u \in \tH^{-1/2}(K)$.
Then $u^\circ \in \tH^{-1/2}(\tilde K) \subset H^{-1/2}(\tilde K)$.
Using (\ref{dual_norm}) and (\ref{1/2-norm}) we have
\beas
     \|u^\circ\|_{H^{-1/2}(\tilde K)}
     & = &
     \sup_{0\not=w\in\tH^{1/2}(\tilde K)}
     {|\<u^\circ,w\>_{0,\tilde K}|\over{\|w\|_{\tH^{1/2}(\tilde K)}}} =
     \sup_{0\not=w\in\tH^{1/2}(\tilde K)}
     {|\<u,w\>_{0,K}|\over{\|w\|_{\tH^{1/2}(\tilde K)}}}
     \\[5pt]
     & = &
     \sup_{0\not=v\in H^{1/2}(K)}\,\,
     \stack{V|_{K}=v}{\sup_{V \in \tH^{1/2}(\tilde K)}}
     {|\<u,V\>_{0,K}|\over{\|V\|_{\tH^{1/2}(\tilde K)}}} =
     \sup_{0\not=v\in H^{1/2}(K)}
     \frac{|\<u,v\>_{0,K}|}
     {\displaystyle{
                    \stack{V|_{K}=v}{\inf_{V\in\tH^{1/2}(\tilde K)}} \|V\|_{\tH^{1/2}(\tilde K)}
                   }
     }
     \\[3pt]
     & \simeq &
     \sup_{0\not=v\in H^{1/2}(K)}
     {|\<u,v\>_{0,K}|\over{\|v\|_{H^{1/2}(K)}}} =
     \|u\|_{\tH^{-1/2}(K)}.
\eeas
Hence, using (\ref{-1/2-norm}) with $u$ replaced by $u^\circ$ and with $K$
replaced by $\tilde K$, we prove (\ref{tilde-1/2-norm}):
\[
  \|u\|_{\tH^{-1/2}(K)} \simeq \|u^\circ\|_{H^{-1/2}(\tilde K)} \simeq
  \Big\|\grad \widetilde{(u^\circ)}_{\tilde K}\Big\|_{0,D}\quad
  \forall u \in \tH^{-1/2}(K).
\]
Then, applying the parallelogram law, integrating by parts, and
making use of notations $2^\circ,\;4^\circ$, we derive (\ref{tilde_-1/2-ip}).
\end{proof}

\begin{remark} \label{rem_edge-norms}
The same arguments as above can be used to find equivalent norms and corresponding
inner products in the Sobolev spaces on any edge $\ell \subset \partial K$.
In particular, using the notation analogous to $3^\circ$ and $4^\circ$, we have
(cf. {\rm (\ref{tilde_1/2-norm}), (\ref{tilde_1/2-ip})})
\beas
     \|u\|_{\tH^{1/2}(\ell)}
     & \simeq &
     \Big\|\grad \widetilde{\widetilde{u^\circ}}\Big\|_{0,K}\quad
     \forall u \in \tH^{1/2}(\ell),
     \\
     \<u,v\>_{\tH^{1/2}(\ell)}
     & = &
     \Big\<\frac{\partial\ttu}{\partial \bn},v\Big\>_{0,\ell} =
     \Big\<u,\frac{\partial\ttv}{\partial \bn}\Big\>_{0,\ell}\quad
     \forall u,v \in \tH^{1/2}(\ell).
\eeas
\end{remark}

The next lemma states the fact that for a constant function $v$ in
(\ref{tilde_-1/2-ip}) the $\tH^{-1/2}(K)$-inner product reduces to the
$L^2(K)$-inner product.

\begin{lemma} \label{lm_tilde_-1/2_ip}
For any $u \in \tH^{-1/2}(K)$ there holds
\[
  \<u,1\>_{\tH^{-1/2}(K)} = \<u,1\>_{0,K}.
\]
\end{lemma}

\begin{proof}
We have by (\ref{tilde_-1/2-ip})
\be \label{lm_tilde1}
    \<u,1\>_{\tH^{-1/2}(K)} = \<u,\varphi|_K\>_{0,K},
\ee
where $\varphi(x)$ ($x = (x_1,x_2,x_3) \in D = K \times (0,1)$)
solves the following mixed problem (see (\ref{tilde_-1/2-ip})
and notations ${1^\circ}$, ${2^\circ}$, $4^\circ$):
find $\varphi \in H^1(D)$ such that
\[
  \Delta \varphi = 0\ \hbox{in $D$},\quad
  \hbox{$\frac{\partial\varphi}{\partial \bnu} = 1$ on $\G_1 = K$},\quad
  \hbox{$\frac{\partial\varphi}{\partial \bnu} = 0$ on $\G_i$ ($i = 2,\ldots,\CI-1$)},\quad
  \varphi = 0\ \hbox{on $\G_{\CI}$}.
\]
It is easy to see that $\varphi = 1 - x_3$.
Then $\varphi|_K = \varphi|_{x_3=0} = 1$ and the assertion follows
from (\ref{lm_tilde1}).
\end{proof}


\end{document}